\documentclass[12pt]{article}
\usepackage{a4wide}

\usepackage[a4paper, margin=2.3cm]{geometry}
\usepackage{amsmath,amssymb,amsthm}
\usepackage{color}
\usepackage{parskip}
\usepackage{hyperref}
\usepackage{parskip}
\usepackage{setspace}
\usepackage{graphicx}

\newtheorem{theorem}{Theorem}[section]

\newtheorem{corollary}[theorem]{Corollary}
\newtheorem{lemma}[theorem]{Lemma}
\newtheorem{definition}[theorem]{Definition}

\newtheorem*{definition*}{Definition}

\begin{document}
\title{Expanding phenomena over higher dimensional matrix rings}

\author{ Nguyen Van The \thanks{VNU University of Science, Vietnam National University, Hanoi. Email: nguyenvanthe\_t61@hus.edu.vn} \and Le Anh Vinh\thanks{Vietnam Institute of Educational Sciences. Email: vinhle@vnies.edu.vn}}
\date{}
\maketitle  
\begin{abstract}

In this paper, we study the expanding phenomena in the setting of higher dimensional matrix rings. More precisely, we obtain a sum-product estimate for large subsets and show that $x(y+z)$, $x+yz$, $xy + z + t$ are moderate expanders over the matrix ring $M_n(\mathbb{F}_q)$.

These results generalize recent results of Y. D. Karabulut, D. Koh, T. Pham, C-Y. Shen, and the second listed author.

\end{abstract}

\section{Introduction}

Let $\mathbb{F}_q$ be a finite field of order $q$ where $q$ is an odd prime power. Given a function $f \colon \mathbb{F}_q^d\to \mathbb{F}_q$, define 
\[f(A, \ldots, A)=\{f(a_1, \ldots, a_d)\colon a_1, \ldots, a_d\in A\},\]
the image of the set $A^d\subset \mathbb{F}_q^d$ under the function $f$. We will use the following definition of expander polynomials, which can be found in \cite{hls}. 
\bigskip
\begin{definition}
Let $f$ be a function from $\mathbb{F}_q^d$ to $\mathbb{F}_q$. 
\begin{itemize}
\item[1.] The function $f$ is called a strong expander with the exponent $\varepsilon>0$ if for all $A\subset \mathbb{F}_q$ with $|A|\gg q^{1-\varepsilon}$, one has $|f(A, \ldots, A)|\ge q-k$ for some fixed constant $k$. 
\item[2.] The function $f$ is called a moderate expander with the exponent $\varepsilon>0$ if for all $A\subset \mathbb{F}_q$ with $|A|\gg q^{1-\varepsilon}$, one has $|f(A, \ldots, A)|\gg q$.
\end{itemize}
\end{definition}
Here and throughout, $X \ll Y$ means that there exists  some absolute constant $C_1>0$ such that $X \leq C_1Y$, $X \gtrsim Y$ means $X\gg (\log Y)^{-C_2} Y$ for some absolute constant $C_2>0$, and  $X\sim Y$ means $Y\ll X\ll Y$.

The study of expander polynomials over finite fields has been of much research interest in recent years. For two variable expanders, Tao \cite{tao} gave a general result that for any polynomial $P(x, y)\in \mathbb{F}_q[x, y]$ that is not one of the forms $Q(F_1(x)+F_2(y))$ and $Q(F_1(x)F_2(y))$ for some one variable polynomials $Q, F_1, F_2$, we have 
\[|P(A, A)|\gg q,\]
under the assumption $|A|\gg q^{1-\frac{1}{16}}$. This implies that such polynomials $P(x,y)$ are moderate expanders with the exponent $\varepsilon=1/16$. 

For three variable expanders over prime fields, Pham, de Zeeuw, and the second listed author \cite{pham} showed that any quadratic polynomial in three variables $P\in \mathbb{F}_p[x,y,z]$ that is not independent on any variable and that does not have the form $G(H(x)+K(y)+L(z))$ for some one variable polynomials $G, H, K, L$ is a moderate expander with exponent $\varepsilon=1/3$. The exponent $1/3$ can be improved in the case of rational function expanders. More precisely, Rudnev, Shkredov, and Stevens \cite{RSS} showed that the function $(xy-z)(x-t)^{-1}$ is a moderate expander with $\varepsilon=17/42$ over prime fields.

Over general finite fields, for three and four variable expanders, there are several families of moderate expanders over finite fields with exponent $\varepsilon = 1/3$ have been constructed by various authors. For example, $(x - y)(z-t)$ by Bennett, Hart, Iosevich, Pakianathan, and Rudnev \cite{bennett}, $xy+zt$ by Hart and Iosevich \cite{ha2}, $x+yz$ by Shparlinski in \cite{shpas}, $x(y+z)$ and $x+(y-z)^2$ by the second listed author \cite{vinh2}. Using methods from spectral graph theory, one can break the exponent $1/3$ by showing that $(x-y)^2+zt$ is a moderate expander with $\varepsilon=3/8$ \cite{vinh2}.

Let $M_n(\mathbb{F}_q)$ be the set of $n \times n$ matrices with entries in $\mathbb{F}_q$. Let $f\colon M_n(\mathbb{F}_q)^d\to M_n(\mathbb{F}_q)$ be a function in $n$ variables. For $A_1, \ldots, A_d\subset M_n(\mathbb{F}_q)$, we define 
\[f(A_1, \ldots, A_d):=\left\lbrace f(a_1, \ldots, a_d)\colon a_i\in A_i, ~1\le i\le d\right\rbrace,\]
the image of the set $A_1 \times A_2 \times \ldots \times A_d \subset \mathbb{F}_q^d$ under the function $f$. 
Similarly, we have the following definition of expanders over the matrix ring $M_n(\mathbb{F}_q),$ which can be found in \cite{vinh}.
\bigskip
\begin{definition}
Let $f$ be a function in $d$ variables from $M_n(\mathbb{F}_q)^d$ to $ M_n(\mathbb{F}_q).$
\begin{itemize}
	\item[1.] The function $f$ is called a \textit{strong expander} over $M_n(\mathbb{F}_q)$ with the exponent $\varepsilon>0$ if for all $A\subset M_n(\mathbb{F}_q)$ with $|A|\gg q^{n^2 - \varepsilon}$, one has $f(A, \ldots, A)\supset GL_n(\mathbb{F}_q)$.
	\item[2.] The function $f$ is called a \textit{moderate expander} over $M_n(\mathbb{F}_q)$ with the exponent $\varepsilon>0$ if for all $A\subset M_n(\mathbb{F}_q)$ with $|A|\gg q^{n^2-\varepsilon}$, one has $|f(A, \ldots, A)|\gg q^{n^2}$. 
	\item[3.] The function $f$ is called a \textit{strong expander} over $SL_n(\mathbb{F}_q)$ with the exponent $\varepsilon > 0$ if for all $A \subset SL_n(\mathbb{F}_q)$ with $|A| \gg q^{n^2-1-\varepsilon},$ one has $f(A,A,\dots,A) \supset GL_n(\mathbb{F}_q).$ 
	\item[4.] The function $f$ is called a \textit{moderate expander} over $M_n(\mathbb{F}_q)$ with the exponent $\varepsilon>0$ if for all $A\subset SL_n(\mathbb{F}_q)$ with $|A|\gg q^{n^2-1-\varepsilon}$, one has $|f(A, \ldots, A)|\gg q^{n^2}$.
\end{itemize}
\end{definition}

The first strong expander polynomial over $M_2(\mathbb{F}_q)$ was given by Ferguson, Hoffman, Luca, Ostafe, and Shparlinski \cite{shpa}. More precisely, they showed that $(x-y)(z-t)$ is a strong expander over $M_2(\mathbb{F}_q)$ with $\varepsilon=1/4$. Similar results in the setting of Heisenberg group over prime fields for small sets were obtained recently by Hegyv\'{a}ri and Hennecart in \cite{HH}. Some generalizations can be found in \cite{koh2, koh3,vinh3,vinh4}.

In \cite{vinh}, Y. D. Karabulut, D. Koh, T. Pham, C-Y. Shen, and the second listed author constructed some families of moderate expanders over $M_2(\mathbb{F}_q)$. More precisely, they showed that $x+yz$, $x(y+z)$ are moderate expanders over $SL_2(\mathbb{F}_q)$ and $M_2(\mathbb{F}_q)$ with the exponent $1/2$ and $xy + z + t$ is strong expander over $M_2(\mathbb{F}_q)$ with exponent $1/4$.

The main purpose of this paper is to extend these results to the setting of $M_n(\mathbb{F}_q)$ for any $n \ge 3$. More precisely, we have the following results.

\begin{theorem} \label{main.theo1}
Let $f(x,y,z)=x+yz$ be a function from $M_n(\mathbb{F}_q) \times SL_n(\mathbb{F}_q) \times SL_n(\mathbb{F}_q)$ to $M_n(\mathbb{F}_q).$ For $A \subset M_n(\mathbb{F}_q)$ and $B,C \subset SL_n(\mathbb{F}_q)$ with $n \ge 3,$ we have
\[ |f(A,B,C)| \gg \min\left\{q^{n^2},q^{2n}|A|,\frac{|A||B|^2|C|^2}{q^{4n^2-4n-2}}, \frac{|B||C|}{q^{n^2-n-1}}\right\}.\]
\end{theorem}
As a direct consequence, the following corollary shows that $f=x+yz$ is a moderate expander over $SL_n(\mathbb{F}_q)$ with the exponent $\varepsilon_n = \frac{n-1}{2}.$ 
\begin{corollary}
Let $f(x,y,z) = x + yz$ be a function from $SL_n(\mathbb{F}_q)^3$ to $M_n(\mathbb{F}_q).$ For $A \subset SL_n(\mathbb{F})$ with $|A| \gg q^{n^2-1-\varepsilon_n}$ where $\varepsilon_n = \frac{n-1}{2},$ we have 
\[ |f(A,A,A)| \gg q^{n^2}.\]
\end{corollary} 
We also have the following corollary which shows that $f=x+yz$ is a moderate expander over $M_n(\mathbb{F}_q)$ with the exponent one.
\begin{corollary}\label{main.coro1}
Let $f(x,y,z)=x+yz$ be a function from $M_n(\mathbb{F}_q)^3$ to $M_n(\mathbb{F}_q).$ For $A \subset M_n(\mathbb{F}_q)$ with $|A| \gg q^{n^2-1},$ we have 
\[ |f(A,A,A)| \gg q^{n^2}.\] 
\end{corollary}

\begin{theorem}\label{main.theo2}
Let $f(x,y,z) = x(y+z)$ be a function from $SL_n(\mathbb{F}_q) \times SL_n(\mathbb{F}_q) \times M_n(\mathbb{F}_q)$ to $M_n(\mathbb{F}_q).$ For $A, B \subset SL_n(\mathbb{F}_q), C \subset M_n(\mathbb{F}_q)$ with $|B| \gg q^{n^2-2}$ and $|C| \gg q^{(n-1)^2},$ we have 
\[ |f(A,B,C)| \gg \min\left\{ q^{n^2}, \frac{|A||B|^2|C|}{q^{3n^2-3n-2}}, \frac{|A||B|}{q^{n^2-n-1}}\right\}. \]
\end{theorem}

As a direct consequence, the following corollary shows that $f=x(y+z)$ is a moderate expander over $SL_n(\mathbb{F}_q)$ with the exponent $\varepsilon=1.$ 
\begin{corollary}
Let $f(x,y,z) = x(y+z)$ be a function from $SL_n(\mathbb{F}_q)^3$ to $M_n(\mathbb{F}_q).$ For $A \subset SL_n(\mathbb{F})$ with $|A| \gg q^{n^2-2},$ we have 
\[ |f(A,A,A)| \gg q^{n^2}.\]
\end{corollary}
We also have the following corollary which shows that $f=x(y+z)$ is a  moderate expander over $M_n(\mathbb{F}_q)$ with the exponent one.
\begin{corollary}\label{main.coro2}
Let $f(x,y,z)=x(y+z)$ be a function from $M_n(\mathbb{F}_q)^3$ to $M_n(\mathbb{F}_q).$ For $A \subset M_n(\mathbb{F}_q)$ with $|A| \gg q^{n^2-1},$ we have 
\[ |f(A,A,A)| \gg q^{n^2}.\] 
\end{corollary}

In the following two theorems, we extend Theorems \ref{main.theo1} and \ref{main.theo2} for arbitrary sets in $M_n(\mathbb{F}_q)$
instead of the special linear group $SL_n(\mathbb{F}_q).$ 

\begin{theorem}\label{theo3}
Let $f(x,y,z) = x(y + z)$ be a function from $M_n(\mathbb{F}_q)^3$ to $M_n(\mathbb{F}_q).$ For $A,B,C \subset M_n(\mathbb{F}_q)$ with $|A||B||C| \gg q^{3n^2-1},$ we have
\[ \left|f(A,B,C)\right| \gg q^{n^2}.\]
\end{theorem}

It is not hard to see that the exponents in Theorem \ref{theo3} and Corollary \ref{main.coro2} are sharp, since one can take $A$ as
the set of zero-determinant matrices in $M_n(\mathbb{F}_q)$ for both cases, and $B = C = M_n(\mathbb{F}_q)$ in case of Theorem \ref{theo3}, we have $|f(A, A, A)| = |A| =o(q^{n^2})$ and $|f(A,B,C)| = |A| =o(q^{n^2}),$ respectively.

\begin{theorem}\label{theo4}
Let $f(x,y,z) = x + yz$ be a function from $M_n(\mathbb{F}_q)^3$ to $M_n(\mathbb{F}_q).$ For $A,B,C \subset M_n(\mathbb{F}_q)$ with $|A||B||C| \gg q^{3n^2-1},$ we have
\[ \left|f(A,B,C)\right| \gg q^{n^2}.\]
\end{theorem}

\begin{theorem}\label{theo5}
 Let $f(x,y,z,t) = xy+z+t$ be a function from $M_n(\mathbb{F}_q)^4$ to $M_n(\mathbb{F}_q).$ Suppose that $A \subset M_n(\mathbb{F}_q)$ and $|A| \gg q^{n^2-1/4},$ then we have
\[ f(A,A,A,A) = M_n(\mathbb{F}_q).\]
\end{theorem}

As a direct consequence from Theorem \ref{theo5}, we have $xy + z + t$ is a strong expander over $M_n(\mathbb{F}_q)$ with the exponent $1/4$. 

In \cite{vinh}, Karabulut, Koh, Pham, Shen, and the second listed author obtained the following sum-product estimate over the matrix ring $M_2(\mathbb{F}_q)$.

\begin{theorem} \textsl{(\cite{vinh})}\label{sp2}
For $A\subset M_2(\mathbb{F}_q)$ with $|A|\gg q^{3}$, we have 
\[\max\left\lbrace |A+A|, |AA|\right\rbrace \gg \min \left\lbrace \frac{|A|^2}{q^{7/2}}, ~q^2|A|^{1/2}\right\rbrace.\]
\end{theorem}

In \cite{thang}, Pham extended this theorem to the setting of $M_n(\mathbb{F}_q)$ for any $n \ge 3$ as follows.

\begin{theorem} \textsl{(\cite{thang})} \label{spt} For $A\subset M_n(\mathbb{F}_q)$ with $n \ge 3$, we have
\begin{itemize}
\item If $|A \cap GL_n(\mathbb{F}_q)| \le |A|/2$ then
\[ \max \{ |A+A|, |AA|\} \gg \min \left\{|A|q, \frac{|A|^3}{q^{2n^2-2n}} \right\}; \]
\item If $|A \cap GL_n(\mathbb{F}_q)| \ge |A|/2$ then
\[  \max \{ |A+A|, |AA|\} \gg \min \left\{ |A|^{2/3}q^{n^2/3}, \frac{|A|^{3/2}}{q^{\frac{n^2}{2}-\frac{1}{4}}} \right\}.\]
\end{itemize}
\end{theorem}

Our second result is the sum-product estimate over the matrix ring $M_n(\mathbb{F}_q)$. 

\begin{theorem}\label{theo6}
For $A \subset M_n(\mathbb{F}_q)$ with $|A|\gg q^{n^2-1},$ we have 
\[ \max\left\{|A + A|,|AA|\right\} \gg \min\left\{\frac{|A|^2}{q^{n^2-1/2}},q^{n^2/2}|A|^{1/2}\right\}.\]
\end{theorem}

Note that, this result is a generalization of Theorem \ref{sp2} and improves the second part of Theorem \ref{spt}.

\section{Proofs of Theorems \ref{main.theo1} and \ref{main.theo2} }

The two following theorems play important roles in the proofs of Theorem \ref{main.theo1} and \ref{main.theo2}. The first theorem was given by Babai, Nikolay and László(\cite{bnl}), and the second theorem was given by T. Pham (\cite{thang}).
 
\begin{theorem}\label{growthSLn}\textsl{(Babai, Nikolay, László, \cite{bnl})}
For $A,B \subset SL_n(\mathbb{F}_q)$ with $n \ge 3,$ we have 
\[ |AB| \gg \min\left\{ q^{n^2-1}, \dfrac{|A||B|}{q^{n^2-n}} \right\}. \] 
\end{theorem}

\begin{theorem}\textsl{(T. Pham, \cite{thang})}\label{sum-growth}
For $A \subset SL_n(\mathbb{F}_q)$ and $B \subset M_n(\mathbb{F}_q)$ with $n \ge 3,$ we have
\[ |A+B| \gg \min\left\{|A|q, \frac{|A|^2|B|}{q^{2n^2-2n-2}}\right\}. \]
\end{theorem}
We now ready to prove Theorem \ref{main.theo1} and Theorem \ref{main.theo2}.

\begin{proof}[\bf Proof of Theorem \ref{main.theo1}]
For $B,C \subset SL_n(\mathbb{F}_q),$ it follows from Theorem \ref{growthSLn}, we have 
\begin{equation}\label{pro.1}
 |BC| \gg \min\left\{ q^{n^2-1}, \frac{|B||C|}{q^{n^2-n}} \right\}.
\end{equation}
Since $B,C$ are subsets in $SL_n(\mathbb{F}_q)$, $BC$ is still a subset in $SL_n(\mathbb{F}_q).$ It follows from Theorem \ref{sum-growth} and \eqref{pro.1}, we obtain
\begin{align*}
|f(A,B,C)|=|A+BC| &\gg \min\left\{ |BC|q, \frac{|A||BC|^2}{q^{2n^2-2n-2}}\right\}  \\
&\gg \min\left\{ q^{n^2}, \dfrac{|B||C|}{q^{n^2-n-1}}, q^{2n}|A|, \dfrac{|A||B|^2|C|^2}{q^{4n^2-4n-2}}\right\}.
\end{align*}
This completes the proof of Theorem \ref{main.theo1}.
\end{proof}

To prove Theorem \ref{main.theo2}, we make use of the following lemma.
\begin{lemma}\label{main-lemma}
Let $\alpha, \beta$ be non-zero elements in $\mathbb{F}_q,$ and $D_\alpha,D_\beta$ be two sets of $n \times n$ matrices of determinants $\alpha$ and $\beta$, respectively. Define
\begin{align*}
D^r_\alpha := \left\{ l_\alpha \cdot x : x \in D_\alpha \right\} \subset SL_n(\mathbb{F}_q)
\end{align*}
and  
\begin{align*}
D^c_\beta := \left\{ y \cdot l_\beta  : y \in D_\beta \right\} \subset SL_n(\mathbb{F}_q)
\end{align*}
where 
\begin{align*}
l_\alpha = \begin{pmatrix}
        \alpha^{-1} & 0  & \dots & 0 \\
         0 & 1 & \dots & 0 \\
        . & . & \dots & .  \\
        0 & 0 & \dots & 1
    \end{pmatrix}, \quad 
l_\beta = \begin{pmatrix}
        \beta^{-1} & 0  & \dots & 0 \\
         0 & 1 & \dots & 0 \\
        . & . & \dots & .  \\
        0 & 0 & \dots & 1
    \end{pmatrix}. 
\end{align*}
Then we have
\[ |D_\alpha D_\beta| = |D^r_\alpha D^c_\beta|.\]
\end{lemma}
\begin{proof}
We first prove that $|D_\alpha D_\beta| = |D^r_\alpha D_\beta|.$ Indeed, let $x, y$ be two matrices in $D_\alpha$ and $D_\beta,$ respectively. Let $z = x \cdot y$ and $x'$ be the corresponding of $x$ in $D^r_\alpha,$ we have 
\[ x' = l_\alpha \cdot x.\]
Observe that $x'\cdot y = (l_\alpha \cdot x) \cdot y = l_\alpha \cdot (x \cdot y) = l_\alpha \cdot z.$ Since $\alpha \neq 0,$ the map $f: z \to l_\alpha\cdot z$ is a one-to-one correspondence between $D_\alpha D_\beta$ and $D^r_\alpha D_\beta.$  

Using the same argument, we can also indicate that there is a correspondence between $D^r_\alpha D_\beta$ and $D^r_\alpha D^c_\beta.$ In other words, we have 
\[ |D_\alpha D_\beta| = |D^r_\alpha D_\beta| = |D^r_\alpha D^c_\beta|.\]
\end{proof}

\begin{proof}[\bf Proof of Theorem \ref{main.theo2}]
We partition the set $B+C$ into $q$ subsets $D_\alpha, \alpha \in \mathbb{F}_q,$ of determinant $\alpha.$
Since $B \subset SL_n(\mathbb{F}_q)$ and $C \subset M_n(\mathbb{F}_q),$ it follows from Theorem \ref{sum-growth}, we have 
\[ |B + C| \gg \min\left\{ |B|q, \frac{|B|^2|C|}{q^{2n^2-2n-2}}\right\} \gg q^{n^2-1} \sim |D_0|,\]
where $|B| \gg q^{n^2-2}$ and $|C| \gg q^{(n-1)^2}.$ Thus, without loss of generality, we assume that
\[ |B+C| \sim \sum_{\alpha \neq 0} |D_\alpha|.\]
Since the matrices in $AD_\alpha$ are of determinant $\alpha,$ $AD_\alpha \cap AD_\beta = \emptyset$ for all $\alpha \neq \beta.$ It follows that 
\[ |A(B+C)| \gg \sum_{\alpha \neq 0} |AD_\alpha|.\]
On the other hand, for each $\alpha \neq 0,$ let 
\begin{align*}
D^c_\alpha := \left\{ x \cdot l_\alpha : x \in D_\alpha \right\} \subset SL_n(\mathbb{F}_q),
\end{align*}
where $l_\alpha$ was defined in Lemma \ref{main-lemma}.
It is clear that $|D^c_\alpha| = |D_\alpha|.$ Moreover, Lemma \ref{main-lemma} gives us $|AD_\alpha| = |AD^c_\alpha|.$ It follows from Theorem \ref{growthSLn}, we have 
\[|AD_\alpha| = |AD^c_\alpha| \gg \min\left\{ q^{n^2-1}, \frac{|A||D^c_\alpha|}{q^{n^2-n}}\right\} = \min\left\{ q^{n^2-1}, \frac{|A||D_\alpha|}{q^{n^2-n}}\right\}. \]
Hence,
\begin{align*}
 |f(A,B,C)|=|A(B+C)| &\gg \sum_{\alpha \neq 0} |AD_\alpha| \\
 &\gg \min\left\{ q^{n^2}, \frac{|A||B+C|}{q^{n^2-n}} \right\} \\
 &\gg \min\left\{ q^{n^2}, \frac{|A||B|^2|C|}{q^{3n^2-3n-2}}, \frac{|A||B|}{q^{n^2-n-1}}\right\}.
\end{align*}
This completes the proof of Theorem \ref{main.theo2}.

\end{proof}

\section{Proofs of Corollaries \ref{main.coro1} and \ref{main.coro2}}

\begin{proof}[\bf Proof of Corollary \ref{main.coro1}]
Firstly, we prove that for $X \subset D_\alpha$ and $Y \subset M_n(\mathbb{F}_q)$ where $D_\alpha$ is the set of matrices of determinant $\alpha \neq 0,$ then 
\begin{equation}\label{det.a}
 |X + Y| \gg \min\left\{ |X|q, \frac{|X|^2|Y|}{q^{2n^2-2n-2}} \right\}.
\end{equation} 
Indeed, let $X^* = \left\{ l_\alpha \cdot x: x \in X \right\}$ and $ Y^* = \left\{ l_\alpha \cdot y: y \in Y \right\}$ with $l_\alpha$ defined in Lemma \ref{main-lemma}, it is easy to check that 
\[ |X + Y| = |X^* + Y^*|.\] 
Applying Theorem \ref{sum-growth}, and using the fact that $|X^*|=|X|$ and $|Y^*| = |Y|,$ we get \eqref{det.a}. 

Since $|A| \gg q^{n^2-1},$ without loss of generality, we can assume that $A \subset GL_n(\mathbb{F}_q).$ Thus, there exist $\alpha \neq 0$ and subset $A' \subset A$ such that all  matrices in $A'$ are of determinant $\alpha$ and $|A'| \gg q^{n^2-2}.$
  
Note that $A'A'$ is the set of matrices of determinant $\alpha^2 \neq 0$, by applying \eqref{det.a} with $X=A'A', Y = A',$  we have 
\[ |A'+A'A'| \gg \min\left\{|A'A'|q, \frac{|A'A'|^2|A'|}{q^{2n^2-2n-2}} \right\}.\]

Let $A^r,A^c$ be the sets of corresponding matrices of determinant $1$ of matrices in $A'$ in the form of Lemma \ref{main-lemma}. It follows from Lemma \ref{main-lemma} and Theorem \ref{growthSLn} that
\[ |A'A'| = |A^rA^c| \gg \min\left\{q^{n^2-1}, \frac{|A^r||A^c|}{q^{n^2-n}}\right\} = \min\left\{ q^{n^2-1}, \frac{|A'|^2}{q^{n^2-n}}\right\}. \]
Therefore,
\[ |A'+A'A'| \gg \min\left\{ q^{n^2},q^{2n}|A'|, \frac{|A'|^2}{q^{n^2-n-1}}, \frac{|A'|^5}{q^{4n^2-4n-2}}\right\} \gg q^{n^2}\]
whenever $|A'| \gg q^{n^2-2},$ which concludes the proof of Corollary \ref{main.coro1}. 
\end{proof}

\begin{proof}[\bf Proof of Corollary \ref{main.coro2}]The proof of Corollary \ref{main.coro2} is almost the same with that of
Corollary \ref{main.coro1}, except that we follow the proof of Theorem \ref{main.theo2} for the set $A'.$

\end{proof}

\section{Sum-product digraph over matrix rings}
In this section, we mimic the study of sum-product digraph over $M_2(\mathbb{F}_q)$ in \cite{vinh} to extend results over $M_n(\mathbb{F}_q).$ 

Let $G$ be a directed graph (digraph) on $n$ vertices where the in-degree and out-degree of each vertex are both $d.$\\[12pt]
Let $A_G$ be the adjacency matrix of $G$, i.e., $a_{ij} = 1$ if there is a directed edge from $i$ to $j$ and zero otherwise. Suppose that $\lambda_1 = d, \lambda_2,...,\lambda_n$ are the eigenvalues of $A_G.$ These eigenvalues can be complex, so we cannot order them, but it is known that $|\lambda_i| \le d$ for all $1 \le i \le n.$ Define $\lambda(G):= \max_{|\lambda_i| \neq d } |\lambda_i|.$ This value is called the second largest eigenvalue of $A_G.$ We say that the $n \times n$ matrix $A$ is normal if $A^tA=AA^t$ where $A^t$ is the transpose of $A.$ The graph $G$ is normal if $A_G$ is normal. There is a simple way to check whenever $G$ is normal or not. Indeed, for any two vertices $x$ and $y,$ let $N^+(x,y)$ be the set of vertices $z$ such that $\overrightarrow{xz},\overrightarrow{yz}$ are edges, and $N^-(x,y)$ be the set of vertices $z$ such that $\overrightarrow{zx},\overrightarrow{zy}$ are edges. By a direct computation, we have $A_G$ is normal if and only if $|N^+(x,y)| = |N^-(x,y)|$ for any two vertices $x$ and $y.$ 

\medskip\noindent A digraph $G$ is called an $(n,d,\lambda)-digraph$ if $G$ has $n$ vertices, the in-degree and out-degree of each vertex are both $d,$ and $\lambda(G) \le \lambda.$ Let $G$ be an $(n,d,\lambda)-digraph.$
We have the following lemma
\begin{lemma} (\cite{van}). Let $G = (V,E)$ be an $(n,d,\lambda)-digraph.$ For any two sets $B,C \subset V,$ we have
\[ \left| e(B,C) - \frac{d}{n}|B||C|\right| \le \lambda\sqrt{|B||C|}\]
where $e(B,C)$ be the number of ordered pairs $(u,w)$ such that $u\in B, w \in C,$ and $\overrightarrow{uw} \in e(G).$
\end{lemma}

We will consider sum-product digraph in the rest of this section. Let $G=(V,E)$ be the sum-product digraph over $M_n(\mathbb{F}_q)$ defined as follows:
\[ V = M_n(\mathbb{F}_q) \times M_n(\mathbb{F}_q),\]
and there is an edge from $(A,C)$ to $(B,D)$ if
\[ A \cdot B = C + D.\]
In the following theorem, we study the $(n,d,\lambda)$ form of this graph.  
\begin{theorem}\label{maintheory}
The sum-product digraph $G$ is an 
\[ (q^{2n^2},q^{n^2},cq^{n^2-\frac{1}{2}})-digraph\]
for some positive constant $c.$
\end{theorem}
To prove Theorem \ref{maintheory},  we first need the following lemma:
\begin{lemma}\label{lemma 1}
Denote $R_m=\left\{A \in M_n\left(\mathbb{F}_q\right) : \mathtt{rank}(A)=m \right\},$ then  
\[ \left| R_m \right| \le d.q^{2mn-m^2}.\]
for some positive constant $d.$
\end{lemma}

\begin{proof}
[Proof of Lemma \ref{lemma 1}.] By definition of $R_m,$ for each $A \in R_m,$ there exist $m$ columns 
\[v_{i_1},v_{i_2},\dots,v_{i_m}, \quad  1\le i_1<i_2<\dots<i_m\le n\] where $v_{i_j} \in \mathbb{F}^n_q$ such that $\left\{v_{i_j}\right\}_{j=1}^m$ is linearly independent and the other columns can be written as a linear combination of $\left\{v_{i_j}\right\}_{j=1}^m.$  \\[10pt]
There are $(q^n-1)$ possibilities for $v_{i_1},$ once we pick $v_{i_1}, v_{i_2}$ has $(q^n -q)$ possibilities, once we have $v_{i_1}$ and $v_{i_2}, v_{i_3}$ has $(q^n-q^2)$ possibilities,$\dots$, $v_{i_m}$ has $(q^n-q^{m-1})$ possibilities. On the other hand, the $n-m$ other columns can be written as a linear combination of $\left\{v_{i_j}\right\}_{j=1}^m,$ each has $q^m$ possibilities. We can choose $m$ of $n$ column vectors, which are linearly independent. Note that two different ways can be just one matrix, so we have
\begin{align*}
|R_m| \le \binom{n}{m} \left(q^n-1\right)\left(q^n-q^2\right)\dots\left(q^n-q^{m-1}\right)\left(q^m\right)^{n-m} \le d.q^{2mn-m^2}.
\end{align*}
for some positive constant $d.$
\end{proof}

\begin{proof}
[Proof of theorem \ref{maintheory}.] It is obvious that the order of $G$ is $q^{2n^2},$ because $\left| M_n(\mathbb{F}_q)\right|= q^{n^2}$ and so $\left|M_n(\mathbb{F}_q) \times M_n(\mathbb{F}_q)\right|=q^{2n^2}.$ Next, we observe that $G$ is a regular digraph of in-degree and out-degree $q^{n^2}.$ Indeed, for any vertex $(A,C) \in V,$ if we choose each matrix $B \in M_n(\mathbb{F}_q),$ there exists a unique $D=A \cdot B - C$ such that 
\[ A \cdot B = C + D.\]
Hence, the out-degree of any vertex in $G$ is $\left|M_n(\mathbb{F}_q)\right|,$ which is $q^{n^2}.$ The same holds for the in-degree of each vertex. \\[10pt]
Let $M$ be the adjacency matrix of $G.$ In the next step, we will bound the second largest eigenvalue of $G.$ To this end, we first need to show that $G$ is a normal digraph. It is known that if $M$ is a normal matrix and $\beta$ is an eigenvalue of $M,$ then the complex conjugate $\overline{\beta}$ is an eigenvalue of $M^t.$ Hence, $|\beta|^2$ is an eigenvalue of $MM^t$ and $M^tM.$ In other words, in order to bound $\beta,$ it is enough to bound the second largest eigenvalue of $MM^t.$ \\[10pt]
We are now ready to show that $M$ is normal. Indeed, let $(A_1,C_1)$ and $A_2,C_2)$ be two different vertices, we now count the of the neighbors $(X,Y)$ such that there are directed edges from $(A_1,C_1)$ and $(A_2,C_2)$ to $(X,Y).$ This number is $N^+((A_1,C_1),(A_2,C_2)).$ We first have
\begin{equation}\label{eq42.1}
A_1X=C_1 + Y, \quad A_2X = C_2 + Y. 
\end{equation}
This implies that 
\begin{equation}\label{eq42.2}
(A_1-A_2)X=C_1 - C_2. 
\end{equation}
Notice that the number of the solutions $X$ to the system \eqref{eq42.2} is the same as that of the solution $(X,Y)$ to the system \eqref{eq42.1}, since if we fix a solution $X$ to \eqref{eq42.2}, then $Y$ in \eqref{eq42.1} is uniquely determined. We now fall into one of the following cases. \\[10pt]
\textbf{Case 1:} If $\det(A_1-A_2) \neq 0,$ then there exists unique $X$ such that $(A_1-A_2)X=C_1-C_2.$ Thus the system \eqref{eq42.1} has only one solution in this case. \\[10pt]
\textbf{Case 2:} If $\det(A_1-A_2) = 0,$ and $\det(C_1-C_2) \neq 0,$ then system $\eqref{eq42.1}$ has no solution. \\[10pt]
\textbf{Case 3:} If $\det(A_1-A_2) = 0,$ and $\det(C_1-C_2)=0,$ then we need to further consider different situations as follows:
\begin{itemize}
\item[1.] If $\mathtt{rank}(A_1-A_2) < \mathtt{rank}(C_1-C_2),$ then system $\eqref{eq42.1}$ has no solution.

\item[2.] If $\mathtt{rank}(A_1-A_2)=0,$ and $\mathtt{rank}(C_1-C_2)=0,$ then we have $A_1=A_2,C_1=C_2.$ This contradicts our assumption that $(A_1,C_1) \neq (A_2,C_2).$ Thus, we can rule out this case.

\item[3.] If $\mathtt{rank}(A_1-A_2)=m,$ and $\mathtt{rank}(C_1-C_2)=k$ where $0 \le k \le m < n.$ As we have known that if the equation $(A_1-A_2)X=C_1-C_2$ has solutions, then the number of its solutions is equal to the number of solutions of $(A_1-A_2)X=0.$ Now, we are ready to count the number of solutions of $(A_1-A_2)X=0$ if the system has solution.

Put $A=A_1-A_2, X = [x_1,x_2,\dots,x_n]$ for some column vectors $x_1,x_2,\dots,x_n \in \mathbb{F}_q^n,$ then $A$ is a matrix in $M_n(\mathbb{F}_q)$ and $\mathtt{rank}(A)=m.$ We just need to estimate the number of solutions of $Ax_1=0$ because the number of solutions of $AX=0$ is equal to $n^{th}$ power of the number of solutions of $Ax_1=0.$ It is known that the set $L$ of all solutions of $Ax_1=0$ is a vector subspace of $\mathbb{F}_q^n$ and has the dimension $\dim L = n - \mathtt{rank}(A) = n-m.$ Hence, we have 
\[ |L|=q^{n-m}.\] 
Therefore, the equation $AX=0$ has $q^{n(n-m)}$ solutions if this equation has solutions. 
\end{itemize}
Since the same argument works for the case of $N^-((A_1,C_1),(A_2,C_2)),$ we obtain the same value for $N^-((A_1,C_1),(A_2,C_2)).$ In short, $M$ is normal.\\[10pt]
As we discussed above, in order to bound the second largest eigenvalue of $M,$ it is enough to bound the second largest value of $MM^t.$ Based on previous calculations, we have
\begin{align*}
MM^t &=(q^{n^2}-1)I+J-E_{0n}-\sum_{0\le m<k<n}E_{mk} - \sum_{0 \le k \le m < n; (m,k) \neq (0,0)}F_{mk} \\ 
&+ \sum_{0 \le k \le m < n; (m,k) \neq (0,0)}\left(q^{n(n-m)}-1\right)H_{mk},
\end{align*}
where $I$ is the identity matrix, $J$ denotes the all-one matrix and the others defined as follows. $E_{0n}$ is the adjacency matrix of the graph $\mathcal{G}_1$ defined as follows: 
\[V(E_{0n}) = M_n(\mathbb{F}_q) \times M_n(\mathbb{F}_q),\]
and there is an edge between $(A_1,C_1)$ and $(A_2,C_2)$ if 
\[ \det(A_1-A_2)=0, \quad \det(C_1-C_2) \neq 0.\]
$E_{mk}$ is the adjacency matrix of the graph $\mathcal{G}_{mk}, 0 \le m < k <n$ defined as follows: 
\[ V(E_{mk}) = M_n(\mathbb{F}_q) \times M_n(\mathbb{F}_q),\]
and there is an edge between $(A_1,C_1)$ and $(A_2,C_2)$ if 
			\[ \mathtt{rank}(A_1-A_2)=m, \quad \mathtt{rank}(C_1-C_2)=k.\]
$F_{mk}$ is the matrix adjacency of the graph $\mathcal{G}_{mk}, 0 \le k \le m<n, (m,k) \neq (0,0)$ defined as follows:
\[ V(F_{mk}) = M_n(\mathbb{F}_q) \times M_n(\mathbb{F}_q),\]
and there is an edge between $(A_1,C_1)$ and $(A_2,C_2)$ if
\[ \mathtt{rank}(A_1-A_2)=m, \quad \mathtt{rank}(C_1-C_2)=k \quad \text{and the equation $AX=0$ has no solution.}\]
$H_{mk}$ is the matrix adjacency of the graph $\mathcal{GP}_{mk}, 0 \le k \le m<n, (m,k) \neq (0,0)$ defined as follows:
\[ V(H_{mk}) = M_n(\mathbb{F}_q) \times M_n(\mathbb{F}_q),\]
and there is an edge between $(A_1,C_1)$ and $(A_2,C_2)$ if
\[ \mathtt{rank}(A_1-A_2)=m, \quad \mathtt{rank}(C_1-C_2)=k \quad \text{and the equation $AX=0$ has solutions.}\]

Suppose $\lambda_2$ is the second largest eigenvalue of $M$ and $\overrightarrow{v_2}$ is the corresponding eigenvector. Since $G$ is a regular graph, we have $J \cdot \overrightarrow{v_2}=0.$ (Indeed, since $G$ is regular, it always has $(1,1,\dots,1)$ as an eigenvector with eigenvalue being its regular-degree. Moreover, since the graph $G$ is connected, this eigenvalue has multiplicity one. Thus any other eigenvectors will be orthogonal to $(1,1,\dots,1)$ which in turns gives us $J\cdot \overrightarrow{v_2}=0).$ Since $MM^t\overrightarrow{v_2}=|\lambda_2|^2\overrightarrow{v_2},$ we get
\begin{align*}
|\lambda|^2 \overrightarrow{v_2} &=(q^{n^2}-1)\overrightarrow{v_2}-E_{0n}\overrightarrow{v_2}-\left(\sum_{0\le m<k<n}E_{mk}\right)\overrightarrow{v_2} - \left(\sum_{0 \le k \le m < n; (m,k) \neq (0,0)}F_{mk}\right)\overrightarrow{v_2} \\ 
&+ \left(\sum_{0 \le k \le m < n; (m,k) \neq (0,0)}\left(q^{n(n-m)}-1\right)H_{mk}\right)\overrightarrow{v_2}.
\end{align*}
Thus $\overrightarrow{v_2}$ is an eigenvector of 
\begin{align*}
(q^{n^2}-1)I-E_{0n}-\sum_{0\le m<k<n}E_{mk} - \sum_{0 \le k \le m < n; (m,k) \neq (0,0)}F_{mk} 
+ \sum_{0 \le k \le m < n; (m,k) \neq (0,0)}\left(q^{n(n-m)}-1\right)H_{mk}.
\end{align*}
Using the Lemma \ref{lemma 1}, one can easily to check that for any $0 \le k<n, 0 \le m < n$ and  $(m,k) \neq (0,0),$ the graph $\mathcal{G}_{mk}$ is $d_{mk}-\text{regular}$ for some $d_{mk}$ where 
\[ d_{mk} \ll q^{2mn-m^2+2nk-k^2} \ll q^{2n^2-2},\]
and the graph $\mathcal{G}_1$ is $d_1-\text{regular},$ where 
\[ d_1 \ll q^{2n^2-1}.\]
Now, we will bound the number $S$ of pairs $(A,C) \in M_n(\mathbb{F}_q)\times M_n(\mathbb{F}_q)$, which satisfies $\mathtt{rank}(A)=m, \mathtt{rank}(C)=k (m \ge k)$ and the equation $AX = C$ has solutions. One considers $AX=C$ where $A$ has some row vectors $a_{i_1},a_{i_2},\dots,a_{i_m} \in \mathbb{F}^n_q,$ which are linearly independent, then each other rows of $A$ can be written as linear combination of $\left\{a_{i_x}\right\}_{x=1}^{m}:$
\[a_j = \alpha_{1j}a_{i_1}+\alpha_{2j}a_{i_2} + \dots + \alpha_{mj} a_{i_m}.\] 
Hence, we get an equivalent equation of $AX=C$ is $A'X=C'$ where $A'$ is obtained from $A$ by keeping $a_{i_1},a_{i_2},\dots,a_{i_m}$ and replacing the other by zero-vector, $C'$ is obtained from $C$ by keeping $c_{i_1},c_{i_2},...,c_{i_m}$ and replacing each other row vectors $c_j$ by 
\[ c'_j = c_j - \alpha_{1j}c_{i_1}-\alpha_{2j}c_{i_2} - \dots - \alpha_{mj} c_{i_m}.\] 
It is clear that $A'X=C'$ has solutions if and only if $c'_j = 0, \forall j \notin \left\{i_1,i_2,\dots,i_m\right\}.$ In other words, $AX=C$ has solutions if only if 
\[  c_j = \alpha_{1j}c_{i_1}+\alpha_{2j}c_{i_2} + \dots + \alpha_{mj} c_{i_m}, \forall j \notin \left\{i_1,i_2,\dots,i_m\right\}. \]
It follows that for each $A,$ the number $N_C$ of matrices $C,$ which satisfies the equation $AX=C$ has solutions only depends on the number of row vectors $c_{i_1},c_{i_2},\dots,c_{i_m}$ with  $\mathtt{rank}[c_{i_1},c_{i_2},\dots,c_{i_m}] = k.$ Using the same argument in proof of Lemma \ref{lemma 1}, we get
\[ N_C \leq \binom{m}{k}\left(q^n-1\right)\left(q^n-q\right)\dots\left(q^n-q^{k-1}\right).q^{k(m-k)}.\]
Hence,
\[ N_C \ll q^{nk+mk-k^2}.\]
Using Lemma \ref{lemma 1} and the boundary of $N_C,$ we have
\[ S \ll q^{2mn-m^2}N_C \ll q^{2mn+nk+mk-m^2-k^2}.\]
Therefore, for any $0\le k \le m$ and $(m,k) \neq (0,0),$ the graph $\mathcal{GP}_{mk}$ is $y_{mk}-\text{regular}$ for some $y_{mk}$ where 
\[ y_{mk} \ll q^{2mn+nk+mk-m^2-k^2}.\]
Note that 
\[q^{n(n-m)}.q^{2mn+nk+mk-m^2-k^2}=q^{2n^2-(m^2+n^2+k^2-kn-km-mn)} \le q^{2n^2-1}, \forall 0 \le k \le m <n;\]
 since eigenvalues of a sum of matrices are bounded by the sum of largest eigenvalue of the summands, we obtain 
\[|\lambda_2| \ll q^{n^2-\frac{1}{2}}\]
which completes the proof of the theorem.
\end{proof}
\section{ Proof of theorem \ref{theo3}, \ref{theo4}, \ref{theo5}, \ref{theo6}}
\begin{proof}
[\bf Proof of Theorem \ref{theo3}.]

Since $|A||B||C| \gg q^{3n^2-1},$ we have $A \gg q^{n^2-1}.$ On the other hand, the number of matrices in $M_n(\mathbb{F}_q)$ with zero-determinant is 
\[ q^{n^2}-(q^n-1)(q^n-q) \dots (q^n-q^{n-1})=q^{n^2-1}+o(q^{n^2-1}).\] 
Thus, without loss of generality, we may assume that $A \subset GL_n(\mathbb{F}_q).$ Define 
\[ U:= \left\{ (a^{-1},b): a \in A, b \in B\right\}, V:=A(B+C) \times C\]
as subsets of vertices in sum-product digraph $G.$ It is clear that $U=|A||B|$ and $V=|C|.|A(B+C)|.$ \\[5pt]
For each vertex $(a^{-1},b)$ in $U,$ it has at least $|C|$ neighbors $(a(b+c),c) \in V.$  Therefore, the number of edges between $U$ and $V$ in the digraph $G$ is at least $|A||B||C|.$ On the other hand, it follows from Theorem \ref{maintheory} and Lemma \ref{lemma 1} that
\[ e(U,V) \ll \frac{|A||B||C||A(B+C)|}{q^{n^2}}+q^{n^2-1/2}(|A||B||C|)^{1/2}\sqrt{|A(B+C)|}.\]
So,
\[ |A||B||C| \ll \frac{|A||B||C||A(B+C)|}{q^{n^2}}+q^{n^2-1/2}(|A||B||C|)^{1/2}\sqrt{|A(B+C)|}.\]
Solving this inequality, we obtain
\[ |A(B+C)| \gg \min\left\{ \frac{|A||B||C|}{q^{2n^2-1}},  q^{n^2}\right\}.\]
and the theorem follows.
\end{proof}

\begin{proof}
[\bf Proof of Theorem \ref{theo4}.]
Define
\[ U := \left\{(b,-a):b \in B, a \in A \right\}, V := C \times (A+BC)\]
as subsets of vertices in the sum-product digraph $G.$ It is clear that $|U| = |A||B|$ and $|V| = |C||A + BC|.$ \\[8pt]
One can check that each vertex $(b,-a)$ in $U$ has at least $|C|$ neighbors $(c,a + b \cdot c) \in  V.$\\[8pt]
This implies that $e(U,V ) \ge |A||B||C|.$ On the other hand, it follows from Theorem \ref{maintheory} and Lemma \ref{lemma 1} that 
\[ e(U,V) \ll \frac{|A||B||C||A+BC|}{q^{n^2}}+q^{n^2-1/2}(|A||B||C|)^{1/2}\sqrt{|A+BC|}.\]
So, 
\[ |A||B||C| \ll \frac{|A||B||C||A+BC|}{q^{n^2}}+q^{n^2-1/2}(|A||B||C|)^{1/2}\sqrt{|A+BC|}.\]
Solving this inequality, we get
\[ |A+BC| \gg \min\left\{ \frac{|A||B||C|}{q^{2n^2-1}},  q^{n^2}\right\}\]
which completes the proof of the theorem.
\end{proof}

\begin{proof}
[\bf Proof of Theorem \ref{theo5}.]
Let $M$ be an arbitrary matrix. We will show that there exist matrices $a_1,a_2,a_3,a_4 \in A$ such that
\[ a_1 \cdot a_2 + a_3 +a_4 = M.\]
Define
\[ U := \left\{ (a_1,-a_3+M): a_1,a_3 \in A\right\},\]
and 
\[ V := \left\{ (a_2,-a_4): a_2,a_4 \in A \right\}\]
as subsets of vertices in the sum-product digraph $G$ over $M_n(\mathbb{F}_q).$\\[8pt]
It is clear that if there is an edge between $U$ and $V,$ then there exist matrices $a_1,a_2,a_3,a_4 \in A$ such that
\[ a_1\cdot a_2 + a_3 + a_4 = M.\]
It follows from Theorem \ref{maintheory} and Lemma \ref{lemma 1} that
\[\left|e(U,V) - \frac{|U||V|}{q^{n^2}}\right| \ll q^{n^2-1/2}\sqrt{|U||V|}.\]
Since $|U|=|V|=|A|^2,$ we have 
\[ e(U,V) > 0 \]
under the condition $|A| \gg q^{n^2-1/4}.$
\end{proof}

\begin{proof}
[\bf Proof of Theorem \ref{theo6}.]
Since $A \gg q^{n^2-1},$ we assume that $A \subset GL_n(\mathbb{F}_q).$ We define
\[ U := \left(A+A\right)\times \left(AA\right), V := \left\{ (a,a \cdot b): a,b \in A\right\}\]
as subsets of vertices in the sum-product graph $G.$ \\[8pt]
It is clear that 
\[ |U| = |AA||A+A|, |V| = |A|^2.\]
Moreover, for each vertex $(a,a \cdot b) \in V,$ it has at least  $|A|$ neighbors $(c+b,a \cdot c)$ in $U.$ Thus the number of edges between $U$ and $V$ is at least $|A|^3.$ \\[8pt]
On the other hand, applying Theorem \ref{maintheory} and Lemma \ref{lemma 1}, we have 
\[e(U,V ) \ll \frac{|U||V|}{q^{n^2}} + q^{n^2-1/2}\sqrt{|U||V|}.\]
Hence, we get 
\[ |A|^3 \ll \frac{|AA||A+A||A|^2}{q^{n^2}}+q^{n^2-1/2}|A|\sqrt{|AA||A+A|}.\]
Set $x =\sqrt{|A + A||AA|}.$ It follows that
\[ |A|x^2 + q^{2n^2-1/2}x - q^{n^2}|A|^2 \ge 0.\]
Solving this inequality, we get
\[ x \ge \frac{-q^{2n^2-1/2}+\sqrt{q^{4n^2-1}+4q^{n^2}|A|^3}}{2|A|}
\]
which implies that
\[ x \gg \min\left\{ \frac{|A|^2}{q^{n^2-1/2}}, q^{n^2/2}|A|^{1/2} \right\}.\]
On the other hand, we observe that 
\[ \max\left\{ |A+A|,|AA|\right\} \ge x\]
which completes the proof of Theorem \ref{theo6}. 
\end{proof}

\bibliographystyle{amsplain}

\begin{thebibliography}{10}

\bibitem{as} N. Alon, J. H. Spencer, \textit{The probabilistic method}, 3rd ed., Wiley-Interscience, 2008.
 
\bibitem{bennett}
M. Bennett, D. Hart, A. Iosevich, J. Pakianathan, M. Rudnev, \textit{Group actions and geometric
combinatorics in} $\mathbb{F}_q^d$, Forum Math., \textbf{29}(1):91--110, 2017.

\bibitem{vinh} Y. Demiroglu Karabulut, D. Koh, T. Pham, C-Y. Shen, L. A. Vinh, \textit{Expanding phenomena over matrix rings}, ArXiv:1803.08357v3 [math.NT] 22 Mar 2019

\bibitem{shpa}
R. Ferguson, C. Hoffman, F. Luca, A. Ostafe, I. Shparlinski, \textit{Some additive combinatorics problems in matrix rings}, Revista Matematica Complutense, \textbf{23}(2) (2010): 501--513.

\bibitem{hls}
D. Hart, L. Li,  C-Y. Shen, \textit{Fourier analysis and expanding phenomena in finite fields}, Proceedings of the American Mathematical Society, \textbf{141}(2) (2013): 461--473.

\bibitem{ha}
D. Hart, A. Iosevich, J. Solymosi, \textit{Sum-product estimates in finite fields via Kloosterman sums}, Int. Math. Res. Not. IMRN, 2007:Art. ID rnm007, 14, 2007.

\bibitem{ha2}
D. Hart, A. Iosevich, \textit{Sums and products in finite fields: an integral geometric viewpoint}. In
Radon Transforms, Geometry, and Wavelets, AMS Contemporary Mathematics 464, pages 129--136.
AMS RI, 2008.

\bibitem{HH}
	N. Hegyv\'{a}ri,  F. Hennecart, \textit{Expansion for cubes in the Heisenberg group}, Forum Mathematicum. Vol. \textbf{30}. No. 1. De Gruyter, 2018.
\bibitem{koh2}
 D. Koh, T. Pham, C.-Y. Shen, L. A. Vinh, \textit{On the determinants and permanents of matrices with restricted entries over prime fields}, accepted in Pacific Journal of Mathematics, 2018. 
 \bibitem{koh3}
  D. Koh, T. Pham, C.-Y. Shen, L. A. Vinh, \textit{Expansion for the product of matrices in groups}, Forum Mathematicum, vol. \textbf{31}, no. 1, pp. 35-48, 2019.

\bibitem{thang} T. Pham, \textit{A sum-product theorem in matrix rings over finite fields}, C. R. Acad. Sci. Paris, Ser. I., Article in press (2019).   

\bibitem{pham} T. Pham, L. A. Vinh, F. De Zeeuw, \textit{Three-variable expanding polynomials and higher-dimensional distinct distances}, Combinatorica (2017): 1--16.

\bibitem{RSS}
M. Rudnev, I. Shkredov,  S. Stevens, \textit{On the energy variant of the sum-product conjecture},  accepted in Revista Matem\'{a}tica Iberoamericana, 2018.


\bibitem{shpas}
I. E. Shparlinski, \textit{On the solvability of bilinear equations in finite fields}, Glasg. Math. J., \textbf{50}(3):523–
529, 2008

\bibitem{soly} J. Solymosi, \textit{Incidences and the spectra of graphs}, in: Combinatorial Number Theory and Additive Group Theory, in: Adv. Courses Math. CRM Barcelona, Birkhuser, Basel, 2009, pp. 299–314.

\bibitem{tao} T. Tao, \textit{Expanding polynomials over finite fields of large characteristic, and a regularity lemma for
definable sets}, Contrib. Discrete Math., 10(1):22--98, 2015.

\bibitem{vinh3} 
L. A. Vinh, \textit{On the permanents of matrices with restricted entries over finite fields},
SIAM Journal on Discrete Mathematics, \textbf{26}(3): 997–1007, 2012.

\bibitem{vinh2}
L. A. Vinh, \textit{On four-variable expanders in finite fields}, SIAM J. Discrete Math., \textbf{27}(4):2038--2048, 2013.


\bibitem{vinh4}
L. A. Vinh, \textit{Singular matrices with restricted entries in vector spaces over finite fields},
Discrete Mathematics, \textbf{312}(2): 413–418, 2013.


\bibitem{van}
V. Vu, 
{\em Sum-product estimates via directed expanders}, 
Math. Res. Lett. 15(2) (2008) 375--388.

\bibitem{bnl} 
L. Babai, N. Nikolay, P. Laszlo, \text{Product growth and mixing in finite groups}, Proceedings of the nineteenth annual ACM–SIAM symposium on Discrete algorithms, Society for Industrial and Applied Mathematics, 2008

\end{thebibliography}

\end{document}